\documentclass[12pt]{amsart}

\usepackage{amssymb, xcolor}%
\usepackage
{amsaddr}
\usepackage{geometry}%
\usepackage{xspace}

\usepackage{graphicx}

\begingroup%
\theoremstyle{plain}
\newtheorem{prop}{Proposition}[section]
\newtheorem{corollary}[prop]{Corollary}%
\newtheorem{lemma}[prop]{Lemma}%
\newtheorem{remark}[prop]{{Remark}}%
\newtheorem{defn}[prop]{{Definition}}%
\newtheorem{question*}{Question}%
\endgroup%

\theoremstyle{plain}
\newtheorem{theorem}[prop]{Theorem}%

\theoremstyle{break}
\newtheorem{open?}[prop]{{Open Question}}

\theoremstyle{break}
\newcommand{\tdlc}{t.d.l.c.\@\xspace}

\def\NN{\mathbb N}
\def\QQ{\mathbb Q}

\def\ZZ{\mathbb Z}

\def\endproof{\hfill$\square$}

\def\ident{e}

\keywords{scale function, finite co-volume, lattice, uniscalar, anisotropic}
\subjclass[2010]{22D05 (Primary)  20E34, 22E40 (Secondary)}
\begin{document}

\title[The scale function and lattices]%
{The scale function and lattices}
\author{G. A. Willis}
\address{School of Mathematical and Physical Sciences,
              The University of Newcastle, University Drive, Building V,
              Callaghan, NSW 2308, Australia, \\
              Tel.: +61-2 4921 5546, 
              Fax: +61-2 4921 6898, 
}
\email{George.Willis@newcastle.edu.au}           
\thanks{Supported by ARC Discovery Project DP120100996}

\date{\today}
\maketitle

\begin{abstract}
It is shown that, given a lattice $H$ in a totally disconnected, locally compact group~$G$, the contraction subgroups in~$G$ and the values of the scale function on $G$ are determined by their restrictions to~$H$. Group theoretic properties intrinsic to the lattice, such as being periodic or infinitely divisible, are then seen to imply corresponding properties of~$G$.  
\end{abstract}

\section{Introduction}
\label{sec:intro}
Lattices in locally compact groups have been studied extensively. In the connected case this reduces to the study of lattices in Lie groups; strong results have been obtained, including the Margulis superrigidity theorem for lattices in higher rank semisimple Lie groups and some of its extensions, see~\cite{Margulis,Raghunathan,Shalom_rigidity}. 

Lattices in specific types of totally disconnected, locally compact (\tdlc) groups have been studied in~\cite{basslubotzky,BurgMozesLattices,l.rank1}, and that a uniform lattice which happens to be a free group controls the scale on an ambient \tdlc group was seen in \cite{Baumcocompact}. It has also  been shown recently that, in contrast with the connected case, compactly generated, simple \tdlc groups need not have lattices, \cite{nolattice,leboudecnolattice}. This note gives further information about how lattices control an ambient \tdlc group $G$ by showing that the contraction subgroups in~$G$ and the values of the scale function on $G$ are determined by their restrictions to the lattice. 

We begin by recalling a few key facts about lattices and the scale function. 
\begin{defn}
\label{defn:finitecovolume}
The closed subgroup~$H$ of the locally compact group~$G$ has \emph{finite co-volume} if $G/H$ supports a finite $G$-invariant measure. If~$H$ has finite co-volume and is discrete with the subspace topology, it is a \emph{lattice} in~$G$.
\end{defn}
The discrete subgroup $H$ has finite co-volume if $G/H$ is compact, that is, if~$H$ is co-compact, see~\cite[Remark 1.11]{Raghunathan}. 

\begin{defn}
\label{defn:scaletidy}
The \emph{scale} on the \tdlc group~$G$ is the function $s_G: G \to \ZZ^+$ defined by
\begin{equation*}
\label{eq:scale}
s_G(x) := \min\left\{ [xVx^{-1} : xVx^{-1}\cap V] \; \colon V\leqslant G,\ V\text{ compact and open}\right\}.
\end{equation*}
A compact open subgroup $V\leq G$ is \emph{tidy for~$x$ in~$G$} if, setting $V_\pm = \bigcap_{n\geq0} x^{\pm n}Vx^{\mp n}$,
\begin{equation*}
\label{eq:tidy}
V = V_+V_- \mbox{ and } V_{++} := \bigcup_{n\geq0} x^nV_+x^{-n} \mbox{ is closed}.
\end{equation*}
\end{defn}
Every \tdlc group has a base of identity neighbourhoods consisting of compact open subgroups, by~\cite{Dant}, \cite[Theorem~II.2.3]{MontZip} or~\cite[Theorem~II.7.7]{HewRoss}. The value $s_G(x)$ is then well-defined because $[xVx^{-1} : xVx^{-1}\cap V]$ is a positive integer. That subgroups tidy for~$x$ always exist is implied by the following.
\begin{theorem}[\cite{Willis94} Definition~p.343 and \cite{Willis01} Theorem~3.1]
\label{thm:tidy}
Let~$G$ be a \tdlc group and $x$ be in $G$. Then the compact open subgroup $V\leq G$ is tidy for~$x$ if and only if
$$
s_G(x) =  [xVx^{-1} : xVx^{-1}\cap V].
$$
\end{theorem}

The scale function on~$G$ is related to the modular function $\Delta : G\to (\QQ^+,\times)$ by the formula $\Delta(x) = s_G(x)/s_G(x^{-1})$. It is well known that the modular function is identically equal to~$1$ if~$G$ contains a lattice, see~\cite[Remark 1.9]{Raghunathan}, and the motivation for this work is to understand how properties of a lattice influence the scale function and related properties of~$G$. The relevant properties of~$G$ are as follows.

\begin{defn}
\label{defn:uniscalar}
Let~$G$ be a \tdlc group.
\begin{enumerate}
\item $G$ is \emph{uniscalar} if the scale function is identically equal to~$1$.
\item The \emph{contraction subgroup}  for $x\in G$ is
$$
{\sf con}(x) := \left\{ g\in G \mid x^ngx^{-n}\to \ident \text{ as }n\to\infty\right\}.
$$
\item $G$ is \emph{anisotropic} if ${\sf con}(x) = \{\ident\}$ for every~$x\in G$.
\end{enumerate}
\end{defn}
\noindent These features of~$G$ are related by the fact that triviality of ${\sf con}(x)$ implies that $s_G(x^{-1}) = 1$, see~\cite[Proposition 3.24]{contraction(aut(tdlcG))}. Hence anisotropic groups are uniscalar, the converse does not hold however. That $G$ should be anisotropic is equivalent to triviality of the \emph{Tits core}, see~\cite[Proposition~3.1]{{CapReidWill1}}.

The observation that ${\sf con}(x^n) = {\sf con}(x)$ for every $n\geq1$ will be useful later.
 
\section{Subgroups with finite co-volume}
\label{sec:fincov}

The proof of the main theorem relies on two results which may be found in~\cite{geometry(tdlcG)} and~\cite{CapReidWill1} respectively but are restated here. The proofs of these results involve an iterative argument that uses the factoring $V = V_+V_-$ of tidy subgroups and the easily verified containments $xV_+x^{-1}\geq V_+$ and $xV_-x^{-1} \leq V$. 

\begin{lemma}[\cite{geometry(tdlcG)}, Lemma~2.4]
\label{lem:scale on double cosets}
Suppose that $V$ is a compact, open subgroup of $G$ that is tidy for $x\in G$ and let $n\geq1$. 
Then $(VxV)^n= V_-x^n V_+$  
and 
$s_G(y)=s_G(x)^n$ for every $y\in (VxV)^n$.
\endproof
\end{lemma}

The following result is established in~\cite{CapReidWill1} for all $y$ in $xV$. The proof that it holds for all $y$~in $VxV$, as claimed here, is the same except that at one point a certain product that is observed to belong to $V$ has an extra factor of $v_1$, where $y = v_1xv_2$.  
\begin{lemma}[{\it c.f.\/} \cite{CapReidWill1}, Lemma~4.1]
\label{lem:conjugcon}
Suppose that $V$ is a compact, open subgroup of $G$ that is tidy for~$x$. Then for every $y\in VxV$ there is $t\in V_+\cap{\sf con}(x^{-1})$ such that $t^{-1}y^ktx^{-k} \in V \mbox{ for every }k\geq0$.
\endproof
\end{lemma}
\begin{remark}
It may happen in Lemma~\ref{lem:conjugcon} that ${\sf con}(x^{-1})$ is trivial, in which case necessarily~$t=\ident$. However any subgroup tidy above for~$x$ then satisfies $xVx^{-1}\leq V$ and it follows that $y^kx^{-k}\in V$. 
\end{remark}

Lemma~\ref{lem:conjugcon} implies the following, by the same argument as given in~\cite{CapReidWill1}. 
\begin{prop}[{\it c.f.\/} \cite{CapReidWill1}, Corollary~4.2]
\label{prop:conjugcon}
Suppose that $V$ is a compact, open subgroup of $G$ that is tidy for~$x$. Then for every $y\in VxV$ there is $t\in V_+\cap{\sf con}(x^{-1})$ such that $t{\sf con}(x)t^{-1} = {\sf con}(y)$.  
\endproof
\end{prop}

We are now ready to prove the main result of this note. 
\begin{theorem}
\label{thm:fincov}
Let $G$ be a totally disconnected, locally compact group and $H$ be a closed subgroup of~$G$ having finite co-volume. Then, for every $x\in G$ there are $n\geq 1$ and $h\in H$ such that 
$$
s_G(x)^n = s_G(h).
$$
Moreover, ${\sf con}(h)$ is conjugate to ${\sf con}(x)$.
\end{theorem}
\begin{proof}
Let~$V$ be a compact, open subgroup of~$G$ that is tidy for~$x$ and let~$\mu$ be a finite $G$-invariant measure on $G/H$. Then $\pi_H(V)$ is an open subset of $G/H$ and hence $\mu(\pi_H(V))>0$. For each $n\geq1$ we have that 
$$
\mu(\pi_H(x^nV)) = \mu(x^n.\pi_H(V)) = \mu(\pi_H(V))
$$ 
because $\mu$ is $G$-invariant. Since~$\mu$ is finite, finite additivity of~$\mu$ implies that there is $n>0$ such that
$$
\pi_H(x^nV)\cap \pi_H(V) \ne\emptyset.
$$
The definition of $\pi_H$ then implies that there is $h\in H$ with $x^nV \cap Vh\ne\emptyset$. Therefore~$h$ belongs to $Vx^nV$. Hence Lemma~\ref{lem:scale on double cosets} implies that $s_G(x)^n = s_G(h)$, and Proposition~\ref{prop:conjugcon} implies that ${\sf con}(h)$ is conjugate to ${\sf con}(x^n)$ which, as remarked in the introduction, is equal to ${\sf con}(x)$.
\end{proof}

It is not true in general that~$s_G(H) = s_G(G)$. For example, the group $G = \QQ_p \times_p \ZZ$ has the co-compact subgroup $H = \QQ_p \times_p n\ZZ$ and 
$$
s_G(G) = \left\{ p^k \mid k\geq0\right\}\mbox{ whereas }s_G(H) = \left\{ p^{nk} \mid k\geq0\right\}.
$$

Several conclusions may be drawn immediately from the theorem.
\begin{corollary}
\label{cor:fincov}
Let $G$ be a \tdlc group and~$H$ be a closed subgroup having finite co-volume. 
\begin{enumerate}
\item If $s_G(h) = 1$ for every $h\in H$, then $G$ is uniscalar. 
\item If ${\sf con}(h)$ is closed for every $h\in H$, then ${\sf con}(x)$ is closed for every $\in G$. 
\item If every element of $H$ has trivial contraction group, then $G$ is anisotropic.
\end{enumerate}
\end{corollary}

Here are two situations in which Corollary~\ref{cor:fincov} applies to extend properties of a lattice to properties of its ambient group.
\begin{prop}
\label{prop:periodic lattice} 
Suppose that the \tdlc group~$G$ has a lattice $H$ with $\langle h\rangle$ finite for every~$h\in H$. Then~$G$ is anisotropic and $\langle x\rangle$ is pre-compact for every~$x\in G$.  
\end{prop}
\begin{proof} 
Every element of~$G$ with finite order has trivial contraction group. Hence~$H$ is anisotropic and it follows by Corollary~\ref{cor:fincov} that $G$ is anisotropic.

Let $x$ be in~$G$ and let $V\leq G$ be tidy for~$x$. Then~$V$ is normalised by~$x$ and, by the argument in the proof of Theorem~\ref{thm:fincov}, there is $n\geq 1$ such that $x^n V \cap H \ne \emptyset$. Choose $h\in x^n V \cap H$. Then $x^n \in \langle h, V\rangle$, which is compact because $h$ has finite order and normalises~$V$. Therefore $\langle x\rangle$ has compact closure.
\end{proof}

Proposition~\ref{prop:periodic lattice} asserts that every element of~$G$ normalises a compact open subgroup. It does not claim that~$G$ has a compact open normal subgroup and that is not true in general, for the group
$$
G = \left\{ x = (x_n)_{n\in\NN} \in \mbox{Sym}(3)^\NN \mid x_n \in \{\ident, (\begin{smallmatrix}
1 & 2 
\end{smallmatrix})\}\mbox{ for almost all }n\right\}
$$
contains the co-compact lattice
$$
H = \left\{ x \in G \mid x_n \in \{\ident, (\begin{smallmatrix}
1 & 2 & 3 
\end{smallmatrix}),(\begin{smallmatrix}
1 & 3 & 2
\end{smallmatrix})\}\mbox{ for all }n\right\}
$$
but has no compact open normal subgroup. This group~$G$ is not compactly generated however, leaving the following question open. 
\begin{question*} Suppose that~$G$ is a compactly generated \tdlc group with a lattice that is a periodic group. Must~$G$ have a compact open normal subgroup? 
\end{question*}

In the next proposition, \emph{infinite divisibility} of the element~$x$ in~$G$ means that there are increasing sequences, $n_k$, of positive integers and, $x_k$, of  elements of $G$ such that $x = x_k^{n_k}$ for every~$k$. 
\begin{prop}
\label{prop:divisible}
Suppose that~$G$ has a lattice $H$ in which every element is infinitely divisible. Then $G$ is uniscalar but need not be anisotropic.
\end{prop}
\begin{proof}
Since $s_G(x^n) = s_G(x)^n$ for every~$n\geq0$, infinite divisibility of~$x$ implies that $s_G(x)=1$. Hence~$H$ is uniscalar and it follows by Corollary~\ref{cor:fincov} that $G$ is uniscalar.

The group $G = C_2^{\QQ}\rtimes \QQ$, where $C_2$ is the group of order~$2$ and $\QQ$ acts on $C_2^{\QQ}$ by translation, has the infinitely divisible lattice $\QQ$ but $\overline{{\sf con}(x)} = C_2^{\QQ}$ for every $x\ne0$ in $\QQ$. Hence $G$ is not anisotropic, thus justifying the last claim. 
\end{proof}

The paper concludes with another question about how properties of a group might depend on a lattice. It is shown in~\cite{Baumcocompact} that, if the \tdlc group~$G$ has a uniform lattice isomorphic to the free group of rank~$k$, then the set of prime divisors of $s_G(G)$ is bounded by a number that depends on~$k$. Note, too, that, by~\cite[Theorem~1]{CapReWi}, that if~$G$ is a compactly generated \tdlc group, there is a finite set $\eta = \eta(G)$ of prime numbers such that the open pro-$\eta$ subgroups of~$G$ form a base of identity neighbourhoods. 
\begin{question*}
Suppose that~$G$ is a \tdlc group with a lattice having~$k$ generators. Is there a bound on $s_G(G)$ and on the local prime content that depends only on~$k$?
\end{question*}

\end{document}